\documentclass[10pt]{amsart}
\usepackage{amsmath,amssymb,amsthm,graphicx,a4wide}
\usepackage[all]{xy}

\newtheorem{theorem}{Theorem}    
\newtheorem{lemma}{Lemma}

\theoremstyle{definition}
\newtheorem{definition}[theorem]{Definition}
\newtheorem{example}[theorem]{Example}

\newtheorem*{remark*}{Remark}
\newcommand{\Z}{\mathbb{Z}}

\newcommand{\Q}{\mathbb{Q}}

\DeclareMathOperator{\Ker}{Ker}

 \makeatletter
    
    \@addtoreset{equation}{section}
  \makeatother

\title[Bi-orderability obstruction for rationally homologically fibered knot]{Alexander polynomial obstruction of bi-orderability for rationally homologically fibered knot groups}
\author{Tetsuya Ito}
\address{Department of Mathematics, Graduate School of Science, Osaka University \\ 1-1 Machikaneyama Toyonaka, Osaka 560-0043, JAPAN}
\email{tetito@math.sci.osaka-u.ac.jp}
\subjclass[2010]{Primary~57M05, Secondary~20F60,06F15}
\urladdr{http://www.math.sci.osaka-u.ac.jp/~tetito/}
\keywords{bi-orderable group, Alexander polynomial}

\begin{document}

\begin{abstract}
We show that if the fundamental group of the complement of a rationally homologically fibered knot in a rational homology 3-sphere is bi-orderable, then its Alexander polynomial has at least one positive real root. Our argument can be applied for a finitely generated group which is an HNN extension with certain properties.
\end{abstract}

\maketitle

\section{Introduction}

A total ordering $\leq_{G}$ on a group $G$ is a \emph{bi-ordering} if $a \leq_{G} b$ implies both $ga \leq_{G} gb$ and $ag \leq_{G} bg$ for all $a,b,g \in G$. A group is called \emph{bi-orderable} if it admits a bi-ordering.

The Alexander polynomial provides a useful criterion for the (non) bi-orderbility. In \cite{cr}, Clay-Rolfsen proved that if the knot $\mathcal{K}$ is \emph{fibered} (actually their argument can be applied for a finitely generated group whose commutator subgroup is finitely generated), then its Alexander polynomial $\Delta_{\mathcal{K}}(t)$ has at least one positive real root. In \cite{cgw} Chiswell-Glass-Wilson showed the same result under the assumption that the group admits a certain two generators, one relator presentation. 

In this note we prove the following (non)-bi-orderability criterion for a rationally homologically fibered knot.

\begin{definition}\cite{gs}
A knot $\mathcal{K}$ in a rational homology 3-sphere $M$ is \emph{rationally homologically fibered} if $\deg \Delta_{\mathcal{K}}(t)=2g(\mathcal{K})$, where $g(\mathcal{K})$ denotes the genus of the knot $\mathcal{K}$. 
\end{definition}

\begin{theorem}
\label{theorem:main}
Let $\mathcal{K}$ be a rationally homologically fibered knot in a rational homology 3-sphere $M$. If the Alexander polynomial $\Delta_{\mathcal{K}}(t)$ has no positive real root, then the knot group $\pi_{1}(M \setminus \mathcal{K})$ is not bi-orderable.
\end{theorem}

Although not all knots are rationally homologically fibered, compared with fibered knots the class of rationally homologically fibered knots are much larger. For example, the alternating knots (in $S^{3}$) are rationally homologically fibered \cite{cro,mu}, and all knots with less than or equal to 11 crossings are rationally homologically fibered, except $11_{n34}$,$11_{n42}$,$11_{n45}$, $11_{n67}$,$11_{n73}$, $11_{n07}$, $11_{n152}$ (in the table Knotinfo \cite{cl}).

\begin{example}
An alternating knot $\mathcal{K}=11a_1$ has the Alexander polynomial $\Delta_{\mathcal{K}}(t)=2-12t+ 30t^2-39t^3+ 30t^4-12t^5+ 2t^6$ which has no positive real root. Thus the fundamental group of its complement is not bi-orderable. ($\mathcal{K}$ is not fibered and \cite{cdn} fails to find a presentation that satisfies the assumption of Chiswell-Glass-Wilson's criterion so they could not detect the non-bi-orderability)
\end{example}

Our argument relies on the rationally homologically fibered condition which in particular forces the Alexander polynomial to be non-trivial. Thus it is interesting to ask whether $\pi_{1}(M \setminus \mathcal{K})$ bi-orderable or not when $\Delta_{\mathcal{K}}(t)=1$.

\section{proof of Theorem}
\label{section:proof}

Let $X= M \setminus \mathcal{K}$ be the knot complement and $G= \pi_{1}(M \setminus \mathcal{K})$ be the knot group.
Let $\pi:\widetilde{X} \rightarrow X$ be the infinite cyclic covering of $X$ which corresponds to the kernel of the abelianization map $\phi: G \rightarrow \Z=\langle t \rangle$. 

Then the homology group of infinite cyclic covering $H_{1}(\widetilde{X};\Q)$ has a structure of $\Q[t,t^{-1}]$ module, where $t$ acts on $\widetilde{X}$ as a deck translation. There exists $p_1(t),\ldots,p_{n}(t) \in \Q[t,t^{-1}]$ and $f \in \Z_{\geq 0}$ such that
\[ H_{1}(\widetilde{X};\Q) \cong \Q[t,t^{-1}]^{f} \oplus \bigoplus_{i=1}^{n} \Q[t,t^{-1}] \slash (p_{i}(t)). \]
The Alexander polynomial $\Delta_{\mathcal{K}}(t)$ is defined by 
\[ \Delta_{\mathcal{K}}(t) = 
\begin{cases} 
p_{1}(t)p_2(t)\cdots p_{n}(t) & (f=0) \\
0 & (f>0).
\end{cases}
\]
Thus $\Delta_{\mathcal{K}}(t) \cdot h =0$ for every $h \in H_{1}(\widetilde{X};\Q)$.

Let $\Sigma$ be a minimum genus Seifert surface of $\mathcal{K}$, and let $Y = M\setminus N(\Sigma)$, where $N(\Sigma)\cong \Sigma \times (-1,1)$ denotes a regular neighborhood of $\Sigma$.

Let $\iota^{\pm}: \Sigma \hookrightarrow \Sigma\times\{\pm 1\} \subset Y$ denotes the inclusion maps. As is well-known, the infinite cyclic covering $\widetilde{X}$ is obtained by gluing infinitely many copies $\{Y_i\}_{i \in \Z}$ of $Y$, where the $i$-th copy $Y_{i}$ and the $(i+1)$-st copy $Y_{i+1}$ are glued by identifying $\iota^{-}(\Sigma) \subset Y_i$ and $\iota^{+}(\Sigma) \subset Y_{i+1}$. 
In the rest of argument, we will always take a base point of $\widetilde{X}$ so that it lies in $Y_{0}$.

For $N\geq 0$, let $Y_{[-N,N]}= \bigcup_{i=-N}^{N} Y_{i} \subset \widetilde{X}$, and let $i_N: Y_0 \hookrightarrow Y_{[-N,N]}$ and $j_N: Y_{[-N,N]} \hookrightarrow \widetilde{X}$ be the inclusion maps. 
We denote the fundamental group $\pi_{1}(Y_{[-N,N]})$ and $\pi_{1}(\widetilde{X}) =\Ker \phi$ by $K_{N}$ and $K$, respectively. Since $Y_{[-N,N]}$ is compact, $K_N$ is finitely generated 
 
Since $\iota^{\pm}_{*}: \pi_{1}(\Sigma) \rightarrow \pi_{1}(X)$ are injective, by van-Kampen theorem it follows that both $(i_{N})_{*}: K_{0} \rightarrow K_N$ and $(j_{N})_{*}: K_{N} \rightarrow K$ are injective. By these inclusion maps we will always regard $K_0$ as a subgroup of $K_{N}$, and $K_N$ as a subgroup of $K$. For $x \in K_0$, we will often write $(i_{N})_{*}(x) \in K_N$ simply by the same symbol $x$, by abuse of notation.

\begin{proof}[Proof of Theorem \ref{theorem:main}]

Assume that $\mathcal{K}$ is rationally homologically fibered, and the Alexander polynomial $\Delta_{\mathcal{K}}(t)$ has no positive real root.

A Theorem of Dubickas \cite{dub} says that a one-variable polynomial $f(t) \in \Q[t,t^{-1}]$ has no positive real root, if and only if there is a non-zero polynomial $g(t) \in \Q[t,t^{-1}]$ such that all the non-zero coefficients of $g(t)f(t)$ are positive. Thus there is a non-zero polynomial $\Delta'(t)$ such that all the non-zero coefficient of $\Delta'(t)\Delta_{\mathcal{K}}(t)$ are positive. We take such $\Delta'(t)$ so that
$\Delta'(t)\Delta_\mathcal{K}(t)=\sum_{i \geq 0} a_{i}t^{i}$ with $a_{0}>0$ and $a_{i}\geq 0$ ($i>0$).

For $x \in K=\pi_{1}(\widetilde{X})$, we denote by $[x] \in H_{1}(\widetilde{X};\Q)$ the homology class represented by $x$. 
Then $[x]=0$ if and only if $x^{r} \in [K,K]$ for some $r>0$.

Let $s \in \pi_{1}(X)$ be an element represented by a meridian of the knot $\mathcal{K}$. Then $t^{i}[x] = [s^{-i}xs^{i}]$.  By definition of the Alexander polynomial, for each $x \in K$ 
\[ \Delta'(t)\Delta_{\mathcal{K}}(t)[x]= \sum_{i\geq 0} a_{i} t^{i}[x] = \sum_{i \geq 0} [s^{-i}x^{a_i}s^{i}] = \left[\prod_{i\geq 0} (s^{-i}x^{a_i}s^{i}) \right] = 0\in H_{1}(\widetilde{X};\Q).\]
This implies that there is $r(x)>0$ such that 
\[ \left( \prod_{i\geq0} (s^{-i}x^{a_i}s^{i}) \right)^{r(x)} \in  [K,K]. \]
Moreover, since $K= \bigcup_{n\geq 0} K_n $, there is $N(x) \in \Z$ such that 
\[  \left( \prod_{i\geq0} (s^{-i}x^{a_i}s^{i}) \right)^{r(x)}  \in [K_{N(x)},K_{N(x)}]. \]

Take a finite symmetric generating set $\mathcal{X}$ of $K_0$. Here symmetric we mean that $x \in \mathcal{X}$ implies $x^{-1} \in \mathcal{X}$.
Let $N = \max\{N(x) \: | \: x \in \mathcal{X}\}$, and let $r$ be the least common multiple of $r(x)$ for $x \in \mathcal{X}$.
Then for every $x \in \mathcal{X}$ we have
\begin{equation}
\label{eqn:cc}  \left( \prod_{i\geq0} (s^{-i}x^{a_i}s^{i}) \right)^{r}  \in [K_{N},K_{N}].
\end{equation}

Now assume to the contrary that, $G$ is bi-orderable. Let $<_{K_N}$ be a bi-ordering on $K_N$ which is the restriction of a bi-ordering of $G$.
Since $K_{N}$ is finitely generated, there is a $<_{K_N}$ convex normal subgroup $C$ of $K_{N}$ such that the quotient group $A_N := K_{N} \slash C$ is a non-trivial, torsion-free abelian group. Then $A_{N}$ has the bi-ordering $<_{A_N}$ coming from $<_{K_{N}}$; $a <_{A_{N}} a'$ if and only if $a=P(k)$, $a'=P(k')$ ($k,k' \in K_{N}$) with $k <_{K_N} k'$, where $P: K_{N} \rightarrow A_N$ denotes the quotient map (see \cite{cr,i} for details on abelian, bi-ordered quotients).

\begin{lemma}
\label{lemma:Qrhf}
Let $q = P\circ (i_N)_{*}: K_0 \stackrel{(i_N)_{*}}{\longrightarrow} K_N \stackrel{P}{\longrightarrow} A_N $.
If both $(\iota^{\pm})_{*}: H_{1}(\Sigma;\Q) \rightarrow H_{1}(Y;\Q)$ are surjections, then $q$ is a surjection.
\end{lemma}
\begin{proof}
By Meyer-Vietoris sequence, the surjectivity of $(\iota^{\pm})_{*}$ shows the surjectivity of $(i_{N})_{*}:H_{1}(Y_{0};\Q) \rightarrow H_{1}(Y_{[-N,N]};\Q)$. Thus $(i_{N})_{*}:H_{1}(Y_{0};\Z) \rightarrow H_{1}(Y_{[-N,N]};\Z)$ is a surjection modulo torsion elements.

On the other hand, $A_N$ is an abelian group so the map $q$ is written as compositions 
\[  K_0=\pi_{1}(Y_0) \rightarrow H_{1}(Y_0;\Z) \stackrel{(i_{N})_{*}}{\longrightarrow} H_{1}(Y_{[-N,N]};\Z) = K_N \slash [K_N,K_N] \rightarrow K_N \slash C = A_{N}.\]
All maps are surjection modulo torsion elements and $A_{N}$ is torsion-free so $q$ is a surjection.
\end{proof}

The following lemma clarifies a role of the rationally homologically fibered assumption (cf. \cite[Proposition 2]{gs}). 

\begin{lemma}
\label{lemma:Qrhf1}
Both $(\iota^{\pm})_{*}: H_{1}(\Sigma;\Q) \rightarrow H_{1}(Y;\Q)$ are surjection if and only if $\mathcal{K}$ is rationally homologically fibered.
\end{lemma}
\begin{proof}
Let $g$ be the genus of $\mathcal{K}$. By Alexander duality, $\dim H_{1}(\Sigma;\Q) = \dim H_{1}(Y;\Q)=2g$. This shows that $(\iota^{\pm})_{*}$ are surjection if and only if $(\iota^{\pm})_{*}$ are isomorphism, that is, they are invertible.

By Meyer-Vietoris sequence, $H_{1}(\widetilde{X};\Q)$  is written as
\[ H_{1}(\widetilde{X};\Q) = \Q[t,t^{-1}] \slash \{ t(\iota^{+})_{*}(h) = (\iota^{-})_{*}(h) \ \ \  \forall h \in H_{1}(\Sigma)\} \]
Thus $\Delta_{\mathcal{K}}(t)$ is equal to the determinant of
$t(\iota^{+})_{*}-(\iota^{-})_{*}: \Q^{2g}= H_{1}(\Sigma;\Q) \rightarrow H_{1}(Y;\Q) \cong \Q^{2g}$. 

If $(\iota^{\pm})_{*}$ are surjective, then they are invertible so
$\Delta_{\mathcal{K}}(t) = \det(t- (\iota^{+})^{-1}_{*}(\iota^{-})_{*}) \det (\iota^{+})$. Since $(\iota^{+})^{-1}_{*}(\iota^{-})_{*}$ is invertible, $\deg \Delta_{\mathcal{K}}(t)= 2g$. Conversely, if $\deg \Delta_{\mathcal{K}}(t)= 2g$ then $\Delta_{\mathcal{K}}(0) = \det((\iota^{-})_{*}) = \det((\iota^{+})_{*})\neq 0$ so both $(\iota^{\pm})_{*}$ are invertible.

\end{proof}

Now we are ready to complete the proof of Theorem.

By Lemma \ref{lemma:Qrhf} and Lemma \ref{lemma:Qrhf1}, if $\mathcal{K}$ is rationally homologically fibered, then $q$ is surjective. Since $\mathcal{X}$ is a symmetric generating set, the surjectivity of $q$ implies that there exists $x \in \mathcal{X}$ such that $1 <_{A_N} q(x)$. By definition of the quotient ordering $<_{A_N}$, $1 <_{K_N} x$. 
The ordering $<_{K_N}$ is the restriction of a bi-ordering of $G$ and $0 \leq a_{i}$ so $1 \leq_{K_N} s^{-i}x^{a_i}s^{i}$. Therefore $1 \leq_{A_N} P(s^{-i}x^{a_i}s^{i})$ for all $i\geq 0$. Since $a_{0}>0$, as a consequence we get 
\[1 <_{A_N} q(x) \leq_{A_N} P \left( \prod_{i\geq 0}(s^{-i}x^{a_i}s^{i}) \right)^{r}. \]

On the other hand, $[K_N,K_N] \subset C $ so (\ref{eqn:cc}) implies 
\[ P\left( \prod_{i\geq 0}(s^{-i}x^{a_i}s^{i})\right)^{r} = 1 \in K_{N} \slash C =A_N.\]
This is a contradiction.

\end{proof}

We state and prove our main theorem for the case that the group is the fundamental group of a knot complement. However, our proof can be applied for finitely generated group represented by a certain HNN extension.

For a finitely generated group $G$ and a surjection $\phi:G \rightarrow \Z=\langle t\rangle$, $H_{1}(\Ker \phi;\Q)$ has a structure of finitely generated $\Q[t,t^{-1}]$-module and the Alexander polynomial $\Delta_{G}^{\phi}(t)$ (with respect to $\phi$) is defined similarly, and have the same property that $\Delta_{G}^{\phi}(t) \cdot h =0 $ for all $ h \in H_{1}(\Ker \phi;\Q)$.

In the proof of Theorem \ref{theorem:main}, besides the assumption that the Alexander polynomial has no positive real roots,
 what we really needed and used can be stated in terms of the groups $\Ker \phi$, $\pi_{1}(\Sigma)$ and $\pi_{1}(Y)$: we used the amalgamated product decomposition
\begin{equation}
\label{eqn:amalgam}
\Ker \phi = \pi_{1}(\widetilde{X}) = \cdots \ast_{\pi_{1}(\Sigma)} \pi_{1}(Y) \ast_{\pi_{1}(\Sigma)} \pi_{1}(Y) \ast_{\pi_{1}(\Sigma)} \pi_{1}(Y) \ast_{\pi_{1}(\Sigma)} \cdots
\end{equation}
having the properties 
\begin{equation}
\label{eqn:fingen}
\pi_{1}(Y) \mbox{ is finitely generated.} 
\end{equation}
\begin{equation}
\mbox{The inclusion } \iota^{\pm}_{*}: \pi_{1}(\Sigma) \rightarrow \pi_1(Y) \mbox{ induce surjections }
\label{eqn:hqsurj}\iota^{\pm}_{*}: H_{1}(\pi_{1}(\Sigma);\Q) \rightarrow H_{1}(\pi_{1}(Y);\Q).
\end{equation}
Note that we used the topological assumption that $\mathcal{K}$ is a rationally homologically fibered knot in a rational homology sphere $M$ only at Lemma \ref{lemma:Qrhf1}, which is used to show the property (\ref{eqn:hqsurj}).

In a language of group theory, the amalgamated product decomposition (\ref{eqn:amalgam}) comes from an expression of $\pi_{1}(M \setminus \mathcal{K})$ as an HNN extension
\[ \pi_{1}(M \setminus \mathcal{K}) = \ast_{\pi_{1}(\Sigma)} \pi_{1}(Y) = \langle t, \pi_{1}(Y)\: |  \: t^{-1}\iota^{+}(s)t= \iota^{-}(s) \ \ (\forall s \in \pi_{1}(\Sigma)) \rangle. \]

In summary, our proof of Theorem \ref{theorem:main} acutally shows the following non-bi-orderbility criterion.

\begin{theorem}\label{theorem:main-gr}
Let $H$ be a finitely generated group and $A$ be a group (not necessarily a finitely generated). Let $\iota^{\pm} : A \rightarrow H$ 
be homomoprhisms such that
\[ (\iota^{\pm})_{*}: H_{1}(A;\Q) \rightarrow H_{1}(H;\Q) \]
are surjective.
Let $G$ be a finitely generated group given by an HNN extension
\[ G =\ast_{A} H = \langle t, H \: |  \: t^{-1}\iota^{+}(a)t= \iota^{-}(a) \  \ (\forall a \in A) \rangle. \]

Let $\phi:G \rightarrow \Z$ is a surjection given by $\phi(t)=1$, $\phi(h)=0$ for all $h \in H$. If the Alexander polynomial $\Delta_{G}^{\phi}(t)$ has no positive real root, then $G$ is not bi-orderable. 
\end{theorem}

\section*{Acknowledgements}

The author thanks for Eiko Kin for valuable comments on earlier draft of the paper. Also, the author thank for Stefan Friedl for pointing out that the proof of Theorem \ref{theorem:main} requires the hypothesis that $K$ is rationally homologically fibered. This work was supported by JSPS KAKENHI Grant Number 15K17540.


\begin{thebibliography}{1}

\bibitem[CL]{cl} J. C. Cha and C. Livingston, Knotinfo: Table of knot invariants. \textsf{http://www.indiana.edu/knotinfo}

\bibitem[CGW]{cgw} I. Chiswell, A. Glass and J. Wilson,
{\em Residual nilpotence and ordering in one-relator groups and knot groups,}
Math. Proc. Camb. Phil. Soc. \textbf{158} (2015), 275--288.


\bibitem[CDN]{cdn} A. Clay and C. Desmarais and P. Naylor,
{\em Testing bi-orderability of knot groups,}
arXiv:1410.5774

\bibitem[CR]{cr} A. Clay and D. Rolfsen, 
{\em Ordered groups, eigenvalues, knots, surgery and L-spaces,}
Math. Proc. Camb. Phil. Soc. \textbf{152} (2012), 115--129.


\bibitem[Cro]{cro} R. Crowell,
{\em Genus of alternating link types,} 
Ann. of Math. (2) \textbf{69} (1959), 258--275.

\bibitem[Dub]{dub} A. Dubickas, 
{\em On roots of polynomials with positive coefficients,}
Manuscripta Math. \textbf{123}(3) (2007) 353--356.

\bibitem[GS]{gs} H. Goda and T. Sakasai,
{\em Homology cylinders and sutured manifolds for homologically fibered knots,}
Tokyo J. Math, \textbf{36} (2013) 85--111.


\bibitem[Ito]{i} T. Ito,
{\em A remark on the Alexander polynomial criterion for the bi-orderability of fibered 3-manifold groups,}
Int. Math. Res. Not. IMRN, (1) (2013) 156--169.

\bibitem[Mur]{mu} K. Murasugi,
{\em On the genus of the alternating knot II},
J. of Math. Soc. Japan. \textbf{10},(1958)  235--248.
.
\end{thebibliography}
\end{document}